\documentclass{amsart}
\usepackage[utf8]{inputenc}

\usepackage{amsmath,amssymb,verbatim}
\usepackage{amsthm}
\usepackage{enumerate}
\usepackage{url}
\usepackage{color}
\usepackage{mathtools}

\newtheorem{theorem}{Theorem}[section]
\newtheorem{lemma}[theorem]{Lemma}
\newtheorem{corollary}[theorem]{Corollary}
\newtheorem{proposition}[theorem]{Proposition}
\newtheorem{fact}[theorem]{Fact}

\theoremstyle{definition}
\newtheorem{definition}[theorem]{Definition}

\newtheorem{remark}[theorem]{Remark}
\newtheorem{question}[theorem]{Question}

\def\e{\varepsilon}
\renewcommand{\emptyset}{\varnothing}
\renewcommand{\epsilon}{\varepsilon}

\def\M{\mathbb M}
\def\R{\mathbb R}
\def\acl{\operatorname{acl}}
\def\dcl{\operatorname{dcl}}
\def\bdd{\operatorname{bdd}}
\def\tp{\operatorname{tp}}
\def\Aut{\operatorname{Aut}}

\newcommand{\abar}{\bar{a}}
\newcommand{\bbar}{\bar{b}}

\newcommand{\ybar}{\bar{y}}

\newcommand{\cL}{\mathcal{L}}

\newcommand{\seq}{\subseteq}
\newcommand{\nv}{\text{-}}
\newcommand{\inv}{^{\nv 1}}
\newcommand{\miff}{\makebox[.4in]{$\Leftrightarrow$}}
\newcommand{\mimp}{\makebox[.4in]{$\Rightarrow$}}

\newcommand{\heq}{\textnormal{heq}}

\newcommand{\eq}{\textnormal{eq}}

\def\Ind{\setbox0=\hbox{$x$}\kern\wd0\hbox to 0pt{\hss$\mid$\hss}
\lower.9\ht0\hbox to 0pt{\hss$\smile$\hss}\kern\wd0}
\def\Notind{\setbox0=\hbox{$x$}\kern\wd0\hbox to 0pt{\mathchardef
\nn=12854\hss$\nn$\kern1.4\wd0\hss}\hbox to
0pt{\hss$\mid$\hss}\lower.9\ht0 \hbox to 0pt{\hss$\smile$\hss}\kern\wd0}
\def\ind{\mathop{\mathpalette\Ind{}}}
\def\nind{\mathop{\mathpalette\Notind{}}}

\newcommand{\inda}{\ind^{\!\!\textnormal{a}}}
\newcommand{\ninda}{\nind^{\!\!\textnormal{a}}}
\newcommand{\indd}{\ind^{\!\!\textnormal{d}}}
\newcommand{\nindd}{\nind^{\!\!\textnormal{d}}}

\newcommand{\indb}{\ind^{\!\!\textnormal{b}}}

\usepackage{stackengine,scalerel}
\def\baseunderline#1{\def\stacktype{L}\def\stackalignment{l}%
  \ThisStyle{\stackon[0pt]{$\SavedStyle#1$}{\let\mathit\relax%
  $\SavedStyle\underline{\phantom{\mathrm{#1}}}$}}}

\makeatletter
\@namedef{subjclassname@2020}{%
  \textup{2020} Mathematics Subject Classification}
\makeatother

\AtBeginDocument{%
   \def\MR#1{}
}

\title[Separation for isometry groups and hyperimaginaries]{Separation for isometric group actions and hyperimaginary independence}

\author[G. Conant]{Gabriel Conant}

\address{Department of Mathematics\\
The Ohio State University\\
Columbus, OH, 43210, USA}
\email{conant.38@osu.edu}

\author[J. Hanson]{James Hanson}

\address{Department of Mathematics\\
University of Maryland\\
College Park, MD 20742, USA}
\email{jhanson9@umd.edu}

\date{February 14, 2022}

\begin{document}

\begin{abstract}
We  generalize P. M. Neumann's Lemma to the setting of isometric actions on metric spaces and use it to prove several results in continuous model theory related to algebraic independence. In particular, we show that algebraic independence satisfies the full existence axiom (which answers a question of Goldbring) and is implied by dividing independence. We also use the relationship between hyperimaginaries and continuous imaginaries  to derive further results that are new even for discrete theories. Specifically, we show that if $\M$ is a monster model of  a discrete or continuous theory, then bounded-closure independence in $\M^{\heq}$ satisfies full existence (which answers a question of Adler) and is implied by dividing independence. 
\end{abstract}


\maketitle

\section{Introduction}

Let $T$ be a complete first-order theory in either discrete or continuous logic, and fix a sufficiently saturated and strongly homogeneous monster model $\M$. The ternary relation of \textbf{algebraic independence} on small subsets of $\M$ is defined by
\[
\textstyle A\inda_C B\miff \acl(AC)\cap\acl(BC)=\acl(C),
\]
where $\acl$ denotes model-theoretic algebraic closure.

In his influential paper, Adler \cite{Adgeo} develops a general theory of invariant ternary relations as mathematical objects in their own right. He also formulates a list of axioms for  \emph{strict independence relations} (see \cite[Section 1]{Adgeo}), and shows that such relations are useful signals for when a theory exhibits a reasonable `notion of independence'. For example, if $T$ is simple then forking independence is a strict independence relation. Part of the relevance of $\inda$ in this study comes from the idea that any useful notion of independence should, at the very least, imply $\inda$.

When $T$ is discrete, $\inda$ satisfies all of the axioms of a strict independence relation except for possibly base monotonicity. The remaining axioms are rather straightforward to check, with the exception of \emph{full existence}, which requires some work to prove directly (see Section \ref{sec:discrete} for details). By definition, an invariant ternary relation $\ind$ on small subsets of $\M$ satisfies \textbf{full existence} if for any $A,B,C\subset\M$ there is some conjugate   $A'\equiv_C A$ such that $A'\ind_C B$. A useful feature of this axiom is that it allows for the construction of lefthand $\ind$-Morley sequences, i.e., $C$-indiscernible sequences $(b_i)_{i<\omega}$ such that $b_i\ind_C b_{<i}$ for all $i<\omega$.

The axioms that are easy to prove for $\inda$ in the discrete setting remain so in the continuous setting, modulo some unsurprising modifications (e.g., \emph{finite character} becomes \emph{countable character}; see \cite[Proposition 2.3.1]{AGK}). 
On the other hand, Adler's discrete proof of full existence for $\inda$ in \cite{Adgeo} does not admit a clear generalization to the continuous setting.  This motivates the following question, which was posed by Goldbring in a 2012 graduate course attended by the first author, and also appears as Question 2.12 in a preprint \cite{AGK2}  by Andrews, Goldbring, and Keisler.

\begin{question}[Goldbring]\label{ques1}
Does $\inda$ satisfy full existence when $T$ is continuous?
\end{question}

Various authors have obtained a positive answer  in certain special cases, although all in unpublished work. We summarize these results at the start of Section \ref{sec:FE}.  

It turns out that a closely related question about hyperimaginaries (in discrete logic) was  asked by Adler in his thesis.  In particular, Adler defines the relation $\indb$ on small subsets of $\M^{\heq}$ so that $A\indb_C B$ if and only if $\bdd(AC)\cap\bdd(BC)=\bdd(C)$. Toward understanding the theory of thorn-forking for hyperimaginaries, he then asks the following question (see \cite[Question A.8]{Adthesis}).

\begin{question}[Adler]\label{ques2}
Does $\indb$ satisfy full existence for small subsets of $\M^{\heq}$?\footnote{To prevent possible confusion, we note that Adler's thesis uses  `existence' instead of `full existence'. The change in terminology to `full existence' occurs in  \cite{Adgeo} which, like many sources, uses `existence' for a much weaker axiom (see Fact \ref{fact:easy}).}
\end{question}

The definition of a hyperimaginary works equally well when $T$ is continuous, and so the previous question can also be asked in that setting. With this perspective, the two questions become essentially equivalent due to the nature of hyperimaginaries in continuous logic (see Section \ref{sec:hyp} for details). 

In this paper, we provide positive answers to the previous questions and prove some further results on the relationship between $\inda$, $\indb$, and dividing independence, denoted $\indd$. In particular, we prove that $\indd$ implies $\inda$ in continuous logic (correcting an erroneous proof from unpublished notes of the first author and  Terry \cite{CoTeCL}). We then use  the continuous perspective on hyperimaginaries  to conclude that $\indd$ implies $\indb$, which appears to be a new result even for discrete theories. Underlying all of these arguments is one key tool of independent interest, namely, a metric space analogue of P. M. Neumann's Lemma, which we prove in Section \ref{sec:PMN}.

\section{The discrete case}\label{sec:discrete}

We first recall a classical result on permutation groups.

\begin{lemma}[P. M. Neumann \cite{NeuPM}]\label{lem:PMN}
Let $X$ be a set and suppose $G$ is a group of permutations of $X$. Let $P,Q\seq X$ be finite subsets such that no point in $P$ has a finite orbit. Then there is some $g\in G$ such that $gP\cap Q=\emptyset$.
\end{lemma}

This lemma was first proved as a corollary of an earlier result called \emph{B. H. Neumman's Lemma} \cite{NeuBH}, which says that if an arbitrary group $G$ is covered by finitely many cosets of (possibly different) subgroups, then at least one of the subgroups involved has finite index. In fact, each  of the lemmas Neumann can be derived as an easy consequence of the other (see \cite{BBMN}).  P. M. Neumann's Lemma has found use in (discrete) model theory  when $X$ is (the universe of) a first-order structure and $G$ is a group of automorphisms  (see also \cite[Corollary 4.2.2]{Hodges}).

As a warm-up for later results in continuous logic, we recall two applications of P. M. Neumann's Lemma in discrete logic having to do with the ternary relations discussed above. The first is a rather short proof that $\inda$ satisfies the full existence axiom. This differs from Adler's proof in \cite[Proposition 1.5]{Adgeo}, which is entirely self-contained and given as a ``hard'' exercise in his thesis  (see \cite[Exercise 1.7$(i)$]{Adthesis}). In retrospect however, the argument in \cite{Adgeo} involves similar combinatorics as Lemma \ref{lem:PMN}. Moreover, as noted in \cite{KrMSE2}, one can  deduce full existence for $\inda$ from \cite[Theorem 6.4.5]{Hodges}, and the connection to Lemma \ref{lem:PMN} is made evident in \cite[Exercise 6.4.6]{Hodges}. Thus the proof we give below is certainly well known. Indeed, a MathSciNet search of papers that cite \cite{NeuPM} reveals similar arguments in \cite{BloMP,HruTJM}, for example.

\begin{proposition}
Let $T$ be a complete discrete theory with monster model $\M$. Then $\inda$ satisfies full existence, i.e., for any small $A,B,C\subset\M$, there is some $A'\equiv_C A$ such that $A'\inda_C B$.
\end{proposition}
\begin{proof}
Fix $A$, $B$, and $C$. Let $B^\dagger=\acl(BC)\backslash\acl(C)$. We want to find $A'\equiv_C A$ such that $\acl(A'C)\cap B^\dagger=\emptyset$. By compactness, it suffices to show that for all finite $A_0\seq\acl(AC)$ and finite $B_0\seq B^\dagger$, there is $A'_0\equiv_C A_0$ such that $A'_0\cap B_0=\emptyset$. So fix such $A_0$ and $B_0$. Then $B_0\cap \acl(C)=\emptyset$, and so no point in $B_0$ has a finite $\Aut(\M/C)$-orbit. By Lemma \ref{lem:PMN}, there is $\sigma\in\Aut(\M/C)$ such that $\sigma(B_0)$ is disjoint from $A_0$. So we can take $A'_0=\sigma\inv(A_0)$. 
\end{proof}

The second application of P. M. Neumann's Lemma is that $\indd$ implies $\inda$. This result is also well-known, but many sources factor the proof through Remark 5.4(3) in \cite{Adgeo}, which was recently discovered to be false \cite{CoKr2}. So we take the opportunity here to spell out the proof from Lemma \ref{lem:PMN}.\footnote{It was pointed out to us later by Alex Kruckman that a similar argument is given in \cite{KrMSE1}.}

\begin{proposition}\label{prop:dadisc}
Let $T$ be a complete discrete theory with monster model $\M$. For any small $A,B,C\subset\M$, if $A\indd_C B$ then $A\inda_C B$.
\end{proposition}
\begin{proof}
Suppose $A\ninda_C B$, and fix $a\in (\acl(AC)\cap\acl(BC))\backslash \acl(C)$. Let $\abar$ be a finite tuple enumerating the $\Aut(\M/BC)$-orbit of $a$, and choose an $\cL_{BC}$-formula $\varphi(x,\bbar)$ defining $\abar$ (so $\varphi(x,\bbar)$ isolates $\tp(a/BC)$). Since $a\not\in\acl(C)$,  every coordinate of $\abar$ has an infinite $\Aut(\M/C)$-orbit. By repeated application of Lemma \ref{lem:PMN}, we can obtain  an infinite sequence $(\abar_i)_{i<\omega}$ of pairwise disjoint tuples such that $\abar_i\equiv_C \abar$ for all $i<\omega$. For each $i<\omega$, choose $\bbar_i$ such that $\abar_i\bbar_i\equiv_C \abar\bbar$. Note that $\abar_i$ enumerates the solutions of $\phi(x,\bbar_i)$. Therefore  $\{\phi(x,\bbar_i):i<\omega\}$ is $2$-inconsistent. Since $\phi(x,\bbar)\in \tp(a/BC)$, we have $a\nindd_C B$. So $A\nindd_C B$ by basic axioms of dividing (see Fact \ref{fact:kdiv} below).
\end{proof}

\section{P. M. Neumann's Lemma for metric spaces}\label{sec:PMN}

 Let $X$ be a metric space. 
 Given $x\in X$ and $\epsilon>0$, let $B_\epsilon(x)$ denote the open ball of radius $\epsilon$ centered at $x$. Given $P,Q\seq X$, we say that $Q$ is an \emph{$\epsilon$-net for $P$} if $P=\bigcup_{x\in Q}B_\epsilon(x)$. A subset $P\seq X$ is \emph{$\epsilon$-bounded} if it has a finite $\epsilon$-net in $X$, and \emph{totally bounded} if it is $\epsilon$-bounded for all $\epsilon>0$. Given $P,Q\seq X$, let $d(P,Q)=\inf\{d(x,y):x\in P,~y\in Q\}$. 
 
Our first form of P. M. Neumann's Lemma has to do with separating finite sets at uniformly chosen distances.

 \begin{lemma}\label{lem:PMc1}
Let $X$ be a metric space and suppose $G$ is a group of isometries of $X$. Let $P,Q\seq X$ be finite sets. Suppose that for all $p\in P$ there is some $\epsilon_p>0$ such that the orbit of $p$ is not $\epsilon_p$-bounded. Then there is some $g\in G$ such that $d(gp,Q)\geq\frac{1}{3}\epsilon_p$ for all $p\in P$.
\end{lemma}
\begin{proof}
The argument is a metric adaptation of the direct proof of P. M. Neumann's Lemma given in \cite{BBMN}. 
We proceed by induction on the size of $P$. We can treat $P=\emptyset$ as a trivial base case. So fix a nonempty set $P$ satisfying the hypotheses of the lemma, and assume that the lemma holds for all finite $P',Q'$ with $|P'|<|P|$.

Fix a point $p\in P$. Since $Gp$ is not $\epsilon_p$-bounded, we can fix some $a\in G$ such that $ap\not\in \bigcup_{y\in Q}B_{\epsilon_p}(y)$. Define $Q_0=\{y\in Q:B_{\frac{1}{3}\epsilon_p}(y)\cap Gp\neq\emptyset\}$ and, 
 for all $y\in Q_0$, choose some $g_y\in G$ such that $d(g_yp,y)<\frac{1}{3}\epsilon_p$. Note that if $x\in P$ then $ax$ does not have an $\epsilon_x$-bounded orbit. So we can apply the induction hypothesis to $P'\coloneqq aP\backslash \{ap\}$ and $Q'\coloneqq Q\cup \bigcup_{y\in Q_0}g_ya\inv Q$, and we obtain some $h\in G$ such that $d(hax,Q')\geq\frac{1}{3}\epsilon_x$ for all $x\in P\backslash\{p\}$. 

If $d(hap,Q)\geq\frac{1}{3}\epsilon_p$ then, setting $g=ha$, we have $d(gx,Q)\geq\frac{1}{3}\epsilon_x$ for all $x\in P$. So we may assume that $d(hap,Q)<\frac{1}{3}\epsilon_p$. Fix $y\in Q$ such that $d(hap,y)<\frac{1}{3}\epsilon_p$. Then $y\in Q_0$. Set $g=a g_y\inv ha$. We show that $d(gx,Q)\geq\frac{1}{3}\epsilon_x$ for all $x\in P$. 

Note that if $x\in P\backslash\{p\}$, then $d(gx,Q)=d(hax,g_ya\inv Q)\geq d(hax,Q')\geq \frac{1}{3}\epsilon_x$. So we just need to show that  $d(gp,Q)\geq\frac{1}{3}\epsilon_p$. Toward a contradiction, suppose there is some $q\in Q$ with $d(gp,q)<\frac{1}{3}\epsilon_p$. Then $d(hap,g_ya\inv q)<\frac{1}{3}\epsilon_p$. Recall that $d(hap,y)<\frac{1}{3}\epsilon_p$ by choice of $y$, and $d(g_yp,y)<\frac{1}{3}\epsilon_p$ by choice of $g_y$. So by the triangle inequality, we have
\[
d(ap,q)=d(g_yp,g_ya\inv q)\leq d(g_yp,y)+d(y,hap)+d(hap,g_ya\inv q)<\epsilon_p.
\]
But this contradicts the initial choice of $a$. 
\end{proof}

Next we use the previous result to prove another version of P. M. Neumann's Lemma for metric spaces, which focuses on uniform separation of compact sets.

\begin{lemma}\label{lem:PMc}
Let $X$ be a metric space and suppose $G$ is a group of isometries of $X$. Let $C\seq X$ be a compact set such that no point in $C$ has a totally bounded orbit. Then there is an $\epsilon>0$ such that for any $\epsilon$-bounded set $D\seq X$ there is some $g\in G$ such that $d(gC,D)\geq\epsilon$.
\end{lemma}

\begin{proof}
Given $a\in C$, fix some $\delta_a>0$ such that $Ga$ is not $\delta_a$-bounded. Since $C$ is compact, there is some finite set $A\seq C$ such that $C\seq \bigcup_{a\in A}B_{\frac{1}{2}\delta_a}(a)$. Set $\epsilon=\frac{1}{18}\min_{a\in A}\delta_a$. We claim that no point in $C$ has a $9\epsilon$-bounded orbit. So fix $x\in C$ and suppose, toward a contradiction, that $F$ is a finite $9\epsilon$-net for $Gx$. Fix $a\in A$ such that $d(x,a)<\frac{1}{2}\delta_{a}$. Then for any $g\in G$, there is some $y\in F$ such that $d(gx,y)<9\epsilon\leq \frac{1}{2}\delta_a$, whence $d(ga,y)\leq d(gx,ga)+d(ga,y)<\delta_a$. So $F$ is a finite $\delta_a$-net for $Ga$, which is a contradiction. 

Now fix an $\epsilon$-bounded set $D\seq X$. We find some $g\in G$ such that $d(gC,D)\geq\epsilon$. Let $P$ be a finite $\epsilon$-net for $C$, and let $Q$ be a finite $\epsilon$-net for $D$. Since $C$ is compact, we may assume $P\seq C$, and so no point in $P$ has a $9\epsilon$-bounded orbit. By Lemma \ref{lem:PMc1}, there is some $g\in G$ such that $d(gP,Q)\geq 3\epsilon$. Toward a contradiction, suppose there are $x\in C$ and $y\in D$ such that $d(gx,y)<\epsilon$. Then there are $p\in P$ and $q\in Q$ such that $\max\{d(p,x),d(q,y)\}<\epsilon$. So 
\[
\textstyle d(gp,q)\leq d(gp,gx)+d(gx,y)+d(y,q)<3\epsilon,
\]
which contradicts the choice of $g$. 
\end{proof}

Note that if $X$ is a set with the discrete metric, then Lemmas \ref{lem:PMc} and \ref{lem:PMc1} both reduce to the original statement of P. M. Neumann's Lemma.

\section{Applications in continuous logic}\label{sec:cont}

\subsection{Review on algebraic closure}

Let $T$ be a complete continuous theory with monster model $\M$.
By definition, a point $a\in\M$ is in the \textbf{algebraic closure} of $C\subset\M$ if there is a compact subset of $\M$ that contains $a$ and is definable over $C$. We will use the following equivalent description of algebraic closure (which essentially quotes  \cite[Exercise 10.8]{BBHU}). 

\begin{fact}\label{fact:acldef}
Given $C\subset\M$ and $a\in\M$, the following are equivalent.
\begin{enumerate}[$(i)$]
\item $a\in\acl(C)$.
\item The $\Aut(\M/C)$-orbit of $a$ is totally bounded.
\item The $\Aut(\M/C)$-orbit of $a$ is compact.  
\item The $\Aut(\M/C)$-orbit of $a$ has bounded cardinality.
\end{enumerate}
\end{fact}

\subsection{Full existence for $\inda$}\label{sec:FE}

In this section, we use the metric space version of P. M. Neumman's Lemma to prove that $\inda$ satisfies full existence in any continuous theory. Let us first summarize the previous literature related to this question. In unpublished notes of the first author and Terry \cite{CoTeCL}, it is shown that $\inda$ satisfies full existence in any continuous theory $T$ for which no complete type forks over its own domain of parameters (so, e.g., if $T$ is simple). In \cite{AGK2}, Andrews, Goldbring, and Keisler prove that if $T$ is a discrete theory in which $\acl$ and $\dcl$ coincide, then $\inda$ satisfies full existence in the continuous randomization $T^R$. Finally, we note an unpublished anonymous result that if  just the metric $d(x,y)$ is stable, then local stability (e.g., as in \cite{BYU}) can be used to prove that $\inda$ satisfies full existence. Altogether, the main result of this section generalizes these various special cases. Moreover, we provide a proof of full existence for $\inda$ that is in direct analogy to the standard proof for discrete logic, and does not require more advanced notions such as forking, randomizations, or local stability. 

In some sources, the question of full existence for $\inda$ is also rephrased in terms of the \emph{extension} axiom. So we take a moment to explain this. Let $T$ be a discrete or continuous theory, and consider the following axioms of an invariant ternary relation $\ind$ on small subsets of $\M$.
\begin{enumerate}[\hspace{5pt}$\ast$]
\item \emph{(extension)} If $A\ind_C B$ and $D\supseteq B$, then there is  $A'\equiv_{BC}A$ such that $A'\ind_C D$.
\item \emph{(existence)} $A\ind_C C$ for any $A,C\subset\M$.
\item \emph{(right monotonicity)} If $A\ind_C B$ and $D\seq B$ then $A\ind_C D$.
\item \emph{(right transitivity)} If $A\ind_C B$ and $A\ind_{BC} D$, then   $A\ind_C D$.
\end{enumerate}
The following facts are straightforward exercises.

\begin{fact}\label{fact:easy}
Let $\ind$ be an invariant ternary relation on small subsets of $\M$.
\begin{enumerate}[$(a)$]
\item If $\ind$ satisfies existence, extension, and right monotonicity, then $\ind$ satisfies full existence.
\item If $\ind$ satisfies full existence and right transitivity, then $\ind$ satisfies extension.
\end{enumerate}
\end{fact}

It is  easy to check that $\inda$ satisfies existence, right monotonicity, and right transitivity. Therefore the question of full existence for $\inda$ is  equivalent to the question of extension. We now prove that this question has a positive answer.

\begin{theorem}\label{thm:a-exist}
In any continuous theory $T$, $\inda$ satisfies full existence (and thus also satisfies extension).
\end{theorem}
\begin{proof}
   Fix $A$, $B$, and $C$. Let and $B^\dagger = \acl(BC)\backslash \acl(C)$. For each $b \in B^\dagger$, let $\e_b > 0$ be a fixed number such that the $\Aut(\M/C)$-orbit of $b$ is not $\e_b$-bounded (such a number exists by Fact \ref{fact:acldef}).  Let $(a_i)_{i<\lambda}$ be an enumeration of $\acl(AC)$, and let $p(x) = \tp(a_{<\lambda}/C)$ in variables $x=(x_i)_{i<\lambda}$. It suffices show that
   \[
   \textstyle p(x) \cup \{d(b,x_i) \geq \frac{1}{3}\e_b : b \in B^\dagger\text{, }i<\lambda\}
   \]
   is finitely satisfiable. Indeed, any realization of this type is of the form $\acl(A'C)$ for some $A'\equiv_C A$. Moreover, if $a\in \acl(A'C)$ then we have $d(b,a)\geq\frac{1}{3}\e_b$ for all $b\in B^{\dagger}$, which implies $A'\inda_C B$.
   
   Fix a finite set $P\seq B^\dagger$ and finitely many indices $i_1<\ldots<i_n<\lambda$. We  will show that the type
   \[
   \textstyle q(x_{i_1},\ldots,x_{i_n})=p(x)|_{x_{i_1},\ldots,x_{i_n}}\cup\{d(b,x_{i_k})\geq\frac{1}{3}\e_b:b\in P,~1\leq k\leq n\}
   \]
   is satisfiable. Let $Q=\{a_{i_1},\ldots,a_{i_n}\}$. For all $b\in P$, the $\Aut(\M/C)$-orbit of $b$ is not $\epsilon_b$-bounded. So by Lemma \ref{lem:PMc1}, there is $\sigma\in\Aut(\M/C)$ such that $d(\sigma(b),Q)\geq\frac{1}{3}\e_b$ for each $b\in P$. Thus $(\sigma\inv(a_{i_1}),\ldots,\sigma\inv(a_{i_n}))$ realizes $q$. 
\end{proof}

\subsection{Bounded closure in $\M^{\heq}$}\label{sec:hyp}

Let $T$ be a discrete or continuous theory. Recall that $\indb$ denotes the hyperimaginary analogue of $\inda$, i.e., for $A,B,C\subset\M^{\heq}$, $A\indb_C B$ if and only if $\bdd(AC)\cap \bdd(BC)=\bdd(C)$. The purpose of this section is to explain how our positive answer to Question \ref{ques1} (full existence for $\inda$ in any continuous theory) implies a positive answer to Question \ref{ques2} (full existence for $\indb$ in $\M^{\heq}$).  The general idea is that $\indb$ is  controlled by $\inda$ in the continuous analogue of $T^{\eq}$. In particular, any hyperimaginary is interdefinable with a bounded set of  hyperimaginaries given by countably type-definable equivalence relations, and any such `countable hyperimaginary' can be treated as an ordinary imaginary in continuous logic. This connection is established in early work of Ben Yaacov \cite{BCat} on \emph{compact abstract theories}, and a focused treatment for continuous logic can be found in the second author's thesis \cite[Chapter 3]{Hanson-thesis}.
 We take the opportunity here to explain the necessary details, which are based on foundations developed by Hart, Kim, and Pillay \cite{HKP} and Ben Yaacov \cite{BCat}.
 
 We start with the following standard fact (see \cite[Fact 1.1]{HKP}), whose proof generalizes easily to continuous logic (and beyond, e.g., \cite[Theorem 2.22]{BCat}). 
  
  \begin{fact}\label{fact:countableEQ}
   Let $E(x,y)$ be a type-definable equivalence relation, and let $(E_i(x,y))_{i<\lambda}$ enumerate all countably type-definable equivalence relations such that $E(x,y)\vdash E_i(x,y)$. Then for any real tuples $a$ and $b$ from $\M$, $E(a,b)$ holds if and only if $E_i(a,b)$ holds for all $i<\lambda$.
  \end{fact}
  
    Let $\M^{\omega\heq}$ denote the set of  hyperimaginaries defined by countably type-definable equivalence relations.
  A consequence of Fact \ref{fact:countableEQ} is that any hyperimaginary in $\M^{\heq}$ is interdefinable with a sequence from $\M^{\omega\heq}$.
Given $A\seq\M^{\heq}$, define $\bdd^\omega(A)\coloneqq \bdd(A)\cap\M^{\omega\heq}$.  The following observations are easy exercises.

  \begin{fact}\label{fact:bom}$~$
  \begin{enumerate}[$(a)$]
  \item If $A\seq\M^{\heq}$ then $\bdd(A)=\dcl^{\heq}(\bdd^\omega(A))$.
  \item If $A\seq\M^{\heq}$ and $\sigma\in\Aut(\M)$, then $\sigma(\bdd^\omega(A))=\bdd^\omega(\sigma(A))$.
  \item If $A,B,C\subset\M^{\heq}$ then $A\indb_C B$ $\Leftrightarrow$ $\bdd^\omega(AC)\cap\bdd^\omega(BC)=\bdd^\omega(BC)$.
  \end{enumerate}
  \end{fact}

  One can also show that if $A\subset\M^{\heq}$ is small, then so is $\bdd^\omega(A)$ (see  the proof of \cite[Proposition 15.18]{Casanovas}).\footnote{In fact, Hart, Kim, and Pillay \cite{HKP} \emph{define} $\bdd(A)$  to be $\bdd^\omega(A)$ in order to ensure that the bounded closure of a small set is small.} Altogether, the question of full existence for $\indb$ in $\M^{\heq}$ reduces to an analogous question in $\M^{\omega\heq}$.

\begin{corollary}\label{cor:omegadown}
$\indb$ satisfies full existence if and only if for all $A,B,C\subset\M^{\omega\heq}$ there is some $A'\equiv_C A$ such that $\bdd^\omega(A'C)\cap\bdd^\omega(BC)=\bdd^\omega(C)$. 
\end{corollary}
\begin{proof}
The forward implication is trivial.  
So assume the latter condition, and fix $A,B,C\subset\M^{\heq}$. By assumption, there is some $A^*\equiv_{\bdd^\omega(C)} \bdd^\omega(AC)$ such that $\bdd^\omega(A^*)\cap \bdd^\omega(BC)=\bdd^\omega(C)$.  Fix $\sigma\in \Aut(\M/\bdd^\omega(C))$ such that $A^*=\sigma(\bdd^\omega(AC))$. Then $\sigma\in\Aut(\M/C)$ by Fact \ref{fact:bom}$(a)$, and so if $A'\coloneqq \sigma(A)$ then $A'\equiv_C A$. Note also that $\bdd^\omega(A^*)=\bdd^\omega(A'C)$
by Fact \ref{fact:bom}$(b)$. So 
\[
\bdd^\omega(A'C)\cap \bdd^\omega(BC)=\bdd^\omega(A^*)\cap \bdd^\omega(BC)=\bdd^\omega(C),
\]
and thus $A'\ind^b_C B$ by Fact \ref{fact:bom}$(c)$. 
\end{proof}

In order to prove full existence for $\indb$, we will construct a  continuous theory $T^*$ having the property that the latter condition of the previous corollary is equivalent to full existence for $\inda$ in $T^*$.  In particular, $T^*$ will be a certain expansion of $T$ by continuous imaginaries. 
 
 Given a  function $f$ on $S_x(T)$ and a model $M\models T$, let $f^{M}$ be the  function on $M^x$ given by $f^{M}(a)\coloneqq f(\tp(a))$ for $a\in M^x$. When $M$ is the monster model $\M$, and there is no possibility of confusion, we omit the superscript and write $f(a)$.

 \begin{definition}
 Let $x$ and $y$ be tuples of variables of the same sort.
 A \textbf{definable pseudo-metric in $x$} is a continuous function $\rho\colon S_{xy}(T)\to\R$ such that for any $M\models T$, $\rho^{M}$ is a pseudometric on $M^x$.
 \end{definition}
 
Next we revisit the construction of definable metrics for hyperimaginary sorts. 
 This is a special case of  \cite[Theorem 2.20]{BCat}, which is stated in the context of compact abstract theories (see \cite[Section 3.1]{BYdgs} for a translation to metric structures). The proof we present here is a simplified version of a proof originally presented in the second author's thesis \cite[Lemma 3.4.4]{Hanson-thesis}.

 \begin{proposition}[Ben Yaacov \cite{BCat,BYdgs}]\label{prop:pseudo-equiv}
 Let $E(x,y)$ be a countably type-definable equivalence relation. Then there is a definable pseudo-metric $\rho$ in $x$ such that for any $a,b\in\M^x$, $E(a,b)$ if and only if $\rho(ab)=0$.
 \end{proposition}
 \begin{proof}
Let $F(xy,x'y')$ be the countably type-definable equivalence relation given by $E(x,x')\wedge y=y'$, where $x'$ and $y'$ are tuples of the same sort as $x$. Define a relation $\equiv_F$ on $S_{xy}(T)$ so that $\tp(ab)\equiv_F \tp(a'b')$ if and only if $ab_F\equiv a'b'_F$.
By basic facts on hyperimaginaries, $\equiv_F$ is a well-defined closed equivalence relation.  Consider the type space $S_F(T)$ corresponding to the quotient of $S_{xy}(T)$ by  $\equiv_F$, and let $\pi\colon S_{xy}(T)\to S_F(T)$ be the quotient map. By assumption on $E$, the set $G\coloneqq \{p\in S_{xy}(T):p\vdash E(x,y)\}$ is a closed $G_\delta$ set in $S_{xy}(T)$. It is  easy to check that $G$ is $\pi$-invariant, and thus $\pi(G)$ is a closed $G_\delta$ set in $S_F(T)$. Since $S_F(T)$ is compact Hausdorff, there is a continuous function $f\colon S_F(T)\to [0,1]$  such that $f(p)=0$ if and only if $p\in \pi(G)$. So $g\coloneqq f\circ\pi\colon S_{xy}(T)\to [0,1]$ is a continuous $\pi$-invariant function with the property that $g(p)=0$ if and only if $p\in\pi\inv(\pi(G))=G$. In other words, $g(\tp(ab))=0$ if and only if $E(a,b)$.   

Now define $\rho\colon S_{xy}(T)\to [0,1]$ such that for $a,b\in\M^x$,
\[
\rho(\tp(ab))\coloneqq \textstyle\sup_z|g(az)-g(bz)|.
\]
Then $\rho$ is well-defined, and it is easy to check that $\rho$ is a pseudo-metric. Using  basic point-set topology, continuity of $g$, and the fact that restriction  from $S_{xyz}(T)$ to $S_{xy}(T)$ is open, one can  show that $\rho$ is continuous. So to finish the proof, it suffices to show that for any $a,b\in\M^x$, $\rho(ab)=0$ if and only if $E(a,b)$ holds.

First suppose $\rho(ab)=0$. This implies $|g(ab)-g(bb)|=0$. Note that $g(bb)=0$ since $E(b,b)$ holds. So $g(ab)=0$, and thus $E(a,b)$ holds. 

Conversely, suppose $E(a,b)$ holds. Then $F(ac,bc)$ holds for all $c\in\M^x$, which implies $\pi(ac)=\pi(bc)$ for all $c\in\M^x$. Therefore we have $g(ac)=g(bc)$ for all $c\in\M^x$ since $g$ is $\pi$-invariant, which implies that $\rho(ab)=0$. 
\end{proof}

Now we define the theory $T^*$ alluded to above.
For each countably type-definable equivalence relation $E(x,y)$, choose a definable pseudometric $\rho_E$ as in Proposition \ref{prop:pseudo-equiv}. Let $T^*$ be the continuous theory that consists of $T$ together with sorts added for quotients by each $\rho_E$. When $x$ is a finite tuple, this construction is described in \cite[Section 11]{BBHU} and \cite[Section 5]{BYU}. For $\omega$-ary product sorts, the construction is more delicate. Some brief remarks on this issue are given in \cite{BYU}, and complete details are provided  in the second author's thesis \cite[Chapter 3]{Hanson-thesis}. In any case, each pseudo-metric $\rho_E$ induces a metric $d_E$ on $\M^x/E$. Using saturation of $\M$ and definability of $\rho_E$, it follows that $\M^x/E$ is a complete metric space with respect to $d_E$. Altogether, this yields a canonical expansion of $\M^{\omega\heq}$ to a monster model of $T^*$ in which the new sort for  $\rho_E$ is interpreted as $(\M^x/E,d_E)$. We now have everything needed to obtain a positive answer to Question \ref{ques2}.
 
\begin{corollary}\label{cor:indb}
In any discrete or continuous theory $T$, $\indb$ satisfies full existence (and thus also satisfies extension).
\end{corollary}
\begin{proof}
By Fact \ref{fact:acldef}, we have $\bdd^\omega(A)=\acl^*(A)$ for any $A\subset\M^{\omega\heq}$. Therefore full existence for $\indb$  follows from Corollary \ref{cor:omegadown} and full existence for $\inda$ in $T^*$, which holds by Theorem \ref{thm:a-exist}. Since $\indb$ satisfies right transitivity, we also get extension  from Fact \ref{fact:easy}$(b)$.
\end{proof}

A pithy summary of this section is that  if $T$ is continuous then Fact \ref{fact:countableEQ} and Proposition \ref{prop:pseudo-equiv} together justify the assertion that $T$ eliminates hyperimaginaries \emph{in the sense of continuous logic}, and so $\bdd$ is controlled by $\acl^{\eq}$.

As a counterpoint to the previous statement, we also note that our use of continuous imaginaries to prove Corollary \ref{cor:indb} is in some sense a formality. Indeed, one could instead work directly with the (abstract) metric space structure induced on $\M^{\omega\heq}$ by Proposition \ref{prop:pseudo-equiv}. Full existence for $\indb$ then follows from Corollary \ref{cor:omegadown} together with an application of Lemma \ref{lem:PMc1} analogous to the proof of Theorem \ref{thm:a-exist}. One would only need to directly check that if $C\subset\M^{\omega\heq}$ then $a\in\bdd^\omega(C)$ if and only if the $\Aut(\M/C)$-orbit of $a$ is totally bounded with respect to the metric on $\M^{\omega\heq}$. But this follows from saturation of $\M$ and definability of the metric.

\subsection{Dividing independence}

In this section, we use the metric version of P. M. Neumann's Lemma to prove that dividing independence implies algebraic independence in continuous logic. A proof of this result was claimed in the original version of \cite{CoTeCL}, but the argument relied on (a continuous adaptation of) the erroneous proof of \cite[Remark 5.4(3)]{Adgeo}. 

Let $T$ be a continuous theory with monster model $\M$.
 Given tuples $\abar$ and $\bbar$ from $\M$, and a set $C\subset\M$, we write $\abar\indd_C \bbar$ (and say \textbf{$\tp(\abar/\bbar C)$ does not divide over $C$}) if for any $C$-indiscernible sequence $(\bbar_i)_{i<\omega}$ such that $\bbar_0=\bbar$, there is some $\abar^*$ such that $\abar^*\bbar_i\equiv_C \abar\bbar$ for all $i<\omega$. If $A$ and $B$ are sets then we write $A\indd_C B$ if $\abar\indd_C \bbar$ holds for some/any tuples $\abar$ and $\bbar$ enumerating $A$ and $B$.

We will make use of the following equivalent characterization of dividing. The discrete analogue is a standard fact (see, e.g., \cite[Corollary 7.1.5, Exercise 7.11]{TeZi}), and the adaptation to continuous logic is routine (see, e.g., \cite{CoTeCL}).

\begin{fact}\label{fact:kdiv} 
Given $C\subset\M$ and tuples $\abar,\bbar$ from $\M$,  $\abar\indd_C \bbar$ if and only if for any $C$-indiscernible sequence $(\bbar_i)_{i<\omega}$, with $\bbar_0\seq \bbar$, there is an $\acl(\abar C)$-indiscernible sequence $(\bbar'_i)_{i<\omega}$ such that $(\bbar'_i)_{i<\omega}\equiv_{\bbar_0 C}(\bbar_i)_{i<\omega}$.
\end{fact}

We now prove the main result of this section.

\begin{theorem}\label{thm:dtoa}
Given $A,B,C\subset\M$, if $A\indd_C B$ then $A\inda_C B$.
\end{theorem}
\begin{proof}
Suppose $A\ninda_C B$. Then there is some $a\in (\acl(AC)\cap\acl(BC))\backslash \acl(C)$. We will show $a\nindd_C B$, which implies $A\nindd_C B$ by Fact \ref{fact:kdiv}.

Let $\bbar$ enumerate $B$, and let $\abar$ enumerate the $\Aut(\M/\bbar C)$-orbit of $a$. Then $\abar$ is compact by Fact \ref{fact:acldef}. Since $a\not\in\acl(C)$, we have $\abar\cap\acl(C)=\emptyset$ and thus no element in $\abar$ has a totally bounded $\Aut(\M/C)$-orbit by Fact \ref{fact:acldef}.  By Lemma \ref{lem:PMc} we can construct a sequence $(\abar_i)_{i<\omega}$ and some $\epsilon>0$ such that:
\begin{enumerate}[$(i)$]
\item $\abar_i\equiv_C\abar$ for all $i<\omega$, and
\item if $i<j<\omega$ then $d(a,a')\geq\epsilon$ for all  $a\in\abar_i$ and $a'\in \abar_j$.
\end{enumerate}
For $i<\omega$, choose $\bbar_i$ such that $\abar_i\bbar_i\equiv_C \abar\bbar$. After replacing $(\abar_i\bbar_i)_{i<\omega}$ by a realization of its EM-type over $C$, we may assume the sequence is $C$-indiscernible, while preserving properties $(i)$ and $(ii)$.  We show that $(\bbar_i)_{i<\omega}$ witnesses $a\nindd_C B$. 

Toward a contradiction, suppose there is some $a^*$ such that $a^*\bbar_i\equiv_C a\bbar$ for all $i<\omega$. Since $a\in \abar$ and $\abar\bbar\equiv_C\abar_i\bbar_i$, we may choose some $a_i\in\abar_i$ such that $a^*\bbar_i\equiv_C a_i\bbar_i$. So  for all $i<\omega$, $a^*$ is in the $\Aut(\M/\bbar_iC)$-orbit of $a_i$, which is precisely $\abar_i$ since $\abar_i\bbar_i\equiv_C\abar\bbar$. Therefore $a^*\in\abar_i$ for all $i<\omega$, contradicting $(ii)$.
\end{proof}

Note that  the previous proof differs slightly from the discrete version in Proposition \ref{prop:dadisc}, which identifies a specific dividing formula. Using the definability of $\abar$ over $BC$, one could  write a more analogous proof that explicitly constructs a \emph{dividing condition} in $\tp(a/BC)$. Specifically, there is an $\cL_C$-formula $\varphi(x,\ybar)$ such that $\varphi(a,\bbar)=0$ and, for some $\delta>0$, $\{\varphi(x,\bbar_i)\leq\delta:i<\omega\}$ is $2$-inconsistent.

 \begin{remark}\label{rem:dtoa}
We briefly sketch another proof of Theorem \ref{thm:dtoa} using full existence for $\inda$ and Erd\H{o}s-Rado in place of Lemma \ref{lem:PMc}.\footnote{This is inspired by a similar argument in discrete logic due to Kruckman \cite{KrMSE2}.}
 
 Assume $A\indd_C B$, and let $\abar$ and $\bbar$ enumerate $A$ and $B$. By full existence for $\inda$ and Erd\H{o}s-Rado, one can construct a $C$-indiscernible sequence $(\bbar_i)_{i<\omega}$ such that $\bbar_0=\bbar$ and $\bbar_i\inda_C\bbar_{<i}$ for all $i<\omega$. By Fact \ref{fact:kdiv} there is an $\acl(\abar C)$-indiscernible sequence $(\bbar'_i)_{i<\omega}$ such that $(\bbar'_i)_{i<\omega}\equiv_{\bbar_0 C}(\bbar_i)_{i<\omega}$. Therefore
 \begin{align*}
 \acl(\abar C)\cap \acl(\bbar C) &= \acl(\abar C)\cap \acl(\bbar'_0 C)&\text{ (since $\bbar'_0=\bbar_0=\bbar$)}\\
 &\seq \acl(\bbar'_1 C)\cap\acl(\bbar'_0C) &\text{ (since $\bbar'_0\equiv_{\acl(\abar C)}\bbar'_1$)}\\
 &=\acl(C) &\text{ (since $\bbar'_0\bbar'_1\equiv_C\bbar_0\bbar_1$ and $\textstyle \bbar_1\inda_C \bbar_0$).}
 \end{align*} 
 So $A\inda_C B$, as desired.
 \end{remark}
 
 A natural question is what Theorem \ref{thm:dtoa} says about the discussion of hyperimaginaries in Section \ref{sec:hyp}. So we note the following conclusion, which appears to be new even in the case that $T$ is discrete.

\begin{corollary}\label{cor:dtob}
In any discrete or continuous theory $T$, $\indd$ implies $\indb$.
\end{corollary}
 \begin{proof}
 Let $T^*$ be as  in Section \ref{sec:hyp}. 
 Given $A\seq\M^{\heq}$, define $\baseunderline{A} \coloneqq \dcl^{\heq}(A)\cap\M^{\omega\heq}$. Then, given $A,B,C\subset\M^{\heq}$, we have 
\begin{align*}
 \textstyle A\indd_C B\text{ (in $\M^{\heq}$)}&\miff \textstyle \baseunderline{AC}\indd_{\baseunderline{C}} \baseunderline{BC} \text{ (in $T^*$)}\\
 &\mimp \textstyle \baseunderline{AC}\inda_{\baseunderline{C}}\baseunderline{BC} \text{ (in $T^*$)}
\miff  A\indb_C B \text{ (in $\M^{\heq}$)}.
\end{align*}
 The middle implication  follows from Theorem \ref{thm:dtoa}. The proof of the remaining equivalences are all routine in light of the various  facts established above.
 \end{proof}
 
The previous corollary is known in the special case that $T$ is a (discrete) \emph{simple} theory (see \cite[Proposition 16.19]{Casanovas}).
We also point out that the alternate proof of Theorem \ref{thm:dtoa} in Remark \ref{rem:dtoa} adapts directly to $\M^{\heq}$ to give an alternate proof of Corollary \ref{cor:dtob} using full existence for $\indb$.

\subsection*{Acknowledgements}
 We  thank Alex Kruckman for helpful discussions. The first author was partially supported by NSF grant DMS-1855503.

%%%%%%%%%%% To ease editing, use normal size for the references:

\end{document}